\documentclass[10pt]{article}
\usepackage{amsmath, amssymb, amsfonts, amsthm, epsfig, float, graphicx, sectsty, stix, comment, todonotes}
\usepackage[all]{xy}

\title{Groups of arbitrary lawlessness growth}
\author{Henry Bradford and Jacob Willis}
\date{}
\newtheorem{thm}{Theorem}[section]
\newtheorem{lem}[thm]{Lemma}
\newtheorem{propn}[thm]{Proposition}
\newtheorem{coroll}[thm]{Corollary}
\newtheorem{defn}[thm]{Definition}
\newtheorem{ex}[thm]{Example}
\newtheorem{notn}[thm]{Notation}

\newtheorem{rmrk}[thm]{Remark}
\newtheorem{qu}[thm]{Question}

\DeclareMathOperator{\im}{im}

\DeclareMathOperator{\SL}{SL}

\DeclareMathOperator{\supp}{supp}

\DeclareMathOperator{\tetr}{tetr}
\DeclareMathOperator{\Wr}{Wr}
\DeclareMathOperator{\PSL}{PSL}

\begin{document}

\maketitle

\begin{abstract}
For a finitely generated lawless group $\Gamma$ and $n \in \mathbb{N}$, 
let $\mathcal{A}_{\Gamma} (n)$ be the minimal positive integer $M_n$ 
such that for all nontrivial reduced words $w$ of length at most $n$ in the free group 
of fixed rank $k \geq 2$, 
there exists $\overline{g} \in \Gamma^k$ of word-length at most $M_n$ 
with $w(\overline{g}) \neq e$. 
For any unbounded nondecreasing function $f : \mathbb{N} \rightarrow \mathbb{N}$ 
satisfying some mild assumptions, 
we construct $\Gamma$ such that the function $\mathcal{A}_{\Gamma}$ is equivalent to $f$. 
Our result generalizes both a Theorem of the first named author, 
who constructed groups for which $\mathcal{A}_{\Gamma}$ is unbounded but grows 
more slowly than any prescribed function $f$, 
and a result of Petschick, 
who constructed lawless 
    groups for which $\mathcal{A}_{\Gamma}$ grows faster than any tower of exponential functions. 
\end{abstract}

\section{Introduction}

For $\Gamma$ a group and $F$ a free group, a nontrivial element $1\neq w\in F$
is a law for $\Gamma$ if it lies in the kernel of every homomorphism
$F \rightarrow \Gamma$.
The group $\Gamma$ is \emph{lawless} if it has no laws.
In \cite{Brad} the first named author introduced a quantitative
version of lawlessness, applicable to finitely generated groups,
and encoded by the \emph{lawlessness growth} function
$\mathcal{A}_{\Gamma}:
\mathbb{N}\rightarrow\mathbb{N}\cup\lbrace \infty \rbrace$.
The group $\Gamma$ is lawless iff $\mathcal{A}_{\Gamma}$ is always finite,
and intuitively, the more slowly $\mathcal{A}_{\Gamma}$ grows,
the stronger the lawlesness of $\Gamma$.


One may seek to estimate $\mathcal{A}_{\Gamma}$ for a given
group $\Gamma$.
For instance, $\mathcal{A}_{\Gamma}$ is bounded iff
$\Gamma$ has a nonabelian free subgroup (\cite{Brad} Theorem 1.1).
One may also ask what growth-types are possible.
It was proved in \cite{Brad} that for any
unbounded nondecreasing function $f$,
there is a finitely generated lawless group $\Gamma$
for which $\mathcal{A}_{\Gamma}$ is unbounded,
but grows more slowly than $f$.
By refining the construction from \cite{Brad},
we prove that almost every nondecreasing function is \emph{equivalent}
to the lawlessness growth of some group.

\begin{thm}[Theorem \ref{fast-cor}] \label{intro-fast-f}
Let $f : \mathbb{N} \rightarrow \mathbb{N}$ be an unbounded non-decreasing function. Suppose there exists some integer $M \geq 2$ where $f(Mn) \geq Mf(n)$ for all $n \in \mathbb{N}$. Then there exists an elementary amenable lawless group $\Gamma$, generated by a finite set $S$ such that for all $n\in\mathbb{N}$, $\mathcal{A}_{\Gamma} ^S (n) \approx f (n)$.
\end{thm}

\begin{thm}[Theorem \ref{slow-ex}] \label{intro-f-slow}
Let $f : \mathbb{N} \rightarrow \mathbb{N}$ be a surjective unbounded non-decreasing function. Suppose there exists some integer $M \geq 2$ where $f(Mn) \leq Mf(n)$ for all $n \in \mathbb{N}$. Then there exists an elementary amenable lawless group $\Gamma$, generated by a finite set $S$ such that for all $n\in\mathbb{N}$, $\mathcal{A}_{\Gamma} ^S (n) \approx f (n)$.
\end{thm} 

Here, ``$\approx$'' is a standard notion of equivalence of growth-functions; 
see Notation \ref{EquivNotn} below. 
In particular there exist finitely generated lawless groups of
arbitrarily fast lawlessness growth.

\begin{thm}[Theorem \ref{faster-growth}] \label{intro-faster-growth}
Let $f : \mathbb{N} \rightarrow \mathbb{N}$ be an unbounded nondecreasing function. There exists an elementary amenable lawless group $\Gamma$, generated by a finite set $S$ such that for all $n\in\mathbb{N}$, 
$\mathcal{A}_{\Gamma} ^S (n) \succeq f (n)$.
\end{thm}

The groups $\Gamma$ constructed in the proofs of Theorems \ref{intro-fast-f}; \ref{intro-f-slow} 
and \ref{intro-faster-growth} are drawn from a family of groups we call 
``sparse wreath products''. These arise as subgroups 
of the unrestricted wreath product $\Delta \Wr \mathbb{Z}$, 
generated by $\mathbb{Z}$ and a finite set $S \subset \Delta ^{\mathbb{Z}}$, 
with the property that translates of the supports of elements of $S$ intersect 
at at most one point in $\mathbb{Z}$. Roughly speaking, 
this ``sparseness'' hypothesis on $S$ implies that the behaviour of word-maps on 
$\Gamma$ reflects that of word-maps on $\Delta$, at scales which we can control 
by carefully choosing the outputs of the functions in $S$. 
The sparse wreath product is a construction which seems flexible enough to be relevant 
to constructing groups with many other kinds of exotic behaviours, 
and we expect it to find applications elsewhere. 

Prior to our Theorem \ref{intro-faster-growth}, the best available construction
of groups of fast lawlessness growth was due to Petschick \cite{Petsc}.
Here, for $k$ a positive integer, $\tetr_k$ is the tetration function
given by $\tetr_k (0) = 1$ and $\tetr_k (n+1) = k^{\tetr_k (n)}$.

\begin{thm}[Petschick] \label{PetMainThm}
There exists a finitely generated lawless group $\Lambda$
such that $\mathcal{A}_{\Lambda} \succeq \tetr_2 \circ \log_8$.
\end{thm}

Theorem \ref{PetMainThm} was based on a study of
lawless $p$-groups of slow \emph{torsion-growth}.
Essentially, in an infinite finitely generated $p$-group,
the lawlessness growth is at least as fast
as the ``inverse'' of the torsion growth function.
A key feature of Petschick's construction is that the groups
constructed in the proof of Theorem \ref{PetMainThm}
are residually finite (being subgroups of the automorphism group
of a locally finite rooted tree).
As Petschick points out, the solution to the restricted Burnside problem
(or rather, an effective version thereof due to Groves and Vaughan-Lee)
yields an explicit \emph{lower} bound on the torsion growth
of a residually finite group.
Therefore, there is an obstruction to constructing lawless
groups of arbitrarily fast lawlessness growth via Petschick's method.
This motivates the following observation, and a question.

\begin{propn}
The groups $\Gamma$ constructed in the proofs of Theorems \ref{intro-fast-f}; \ref{intro-f-slow} 
and \ref{intro-faster-growth}
are not residually finite.
\end{propn}
\begin{proof}
By Proposition \ref{not-rf}.
\end{proof}

\begin{qu}
For which functions $f : \mathbb{N} \rightarrow \mathbb{N}$
does there exist a finitely generated residually finite lawless group $\Gamma$
such that $\mathcal{A}_{\Gamma} \approx f$?
\end{qu}

\section{Preliminaries}

Let $\Gamma$ be a group with generating set $S$, and $\lvert \cdot \rvert_S : \Gamma \rightarrow \mathbb{N}$ denote the corresponding word-length function. Let $F_k$ be the free group of rank $k$. We denote by $\lvert \cdot \rvert$ the word length on $F_k$ induced by a fixed free basis $X=\{x_1,...,x_k\}$. For a word $w \in F_k$ and $g_1,...,g_k \in \Gamma$ we will use $w(g_1, ... , g_k)$ to denote the evaluation of the word $w$ under the substitution $x_i \mapsto g_i$. The set of \emph{laws} for $\Gamma$ in $F_k$ is defined to be 
\begin{equation*}
 L_k (\Gamma) \coloneq \bigcap_{\pi \in \mathrm{Hom}(F_k, \Gamma)} \ker(\pi) \setminus \{ 1 \}
\end{equation*}
and the set of \emph{non-laws} for $\Gamma$ in $F_k$ is defined to be $N_k (\Gamma) \coloneq F_k \setminus (L_k (\Gamma) \cup \{1\})$. A group is called \emph{lawless} if for some $k \geq 2$ (equivalently for all $k \geq 1$) that $L_k (\Gamma)$ is empty. In this case $N_k (\Gamma) = F_k \setminus \{1\}$.

For $w \in N_k (\Gamma)$, we define the \emph{complexity} of $w$ in $\Gamma$ 
(with respect to $S$) to be
\begin{equation*}
\chi_{\Gamma} ^S (w) 
\coloneq \min \big\lbrace \sum_{i=1} ^k \lvert g_i \rvert_S 
: g_i \in \Gamma , w(g_1, ... , g_k) \neq e \big\rbrace
\end{equation*}
and for $W \subseteq N_k (\Gamma)$, the \emph{$W$-lawlessness growth function} of $\Gamma$ (with respect to $S$) is defined to be
\begin{equation*}
\mathcal{A}_{\Gamma,W} ^S (l) 
\coloneq \max \big\lbrace \chi_{\Gamma} ^S (w) 
: w \in W; \lvert w \rvert \leq l \big\rbrace\text{.}
\end{equation*}

If $k \geq 2$ and $W = F_k \setminus \lbrace 1 \rbrace$ (which is only possible when $\Gamma$ is lawless) we denote $\mathcal{A}_{\Gamma,F_k \setminus \lbrace 1 \rbrace} ^S$ 
by  $\mathcal{A}_{\Gamma} ^S$ and refer to it simply 
as the \emph{lawlessness growth of $\Gamma$} 
(with respect to $S$). The following equivalence relation will show why the choice of $k \geq 2$ does not matter.

\begin{notn} \textnormal{\cite[Notation 2.2]{Brad}} \label{EquivNotn}
For nondecreasing functions 
$F_1 , F_2 : \mathbb{N} \rightarrow \mathbb{N}$ 
we write $F_1 \preceq F_2$ if there exists $K \in \mathbb{N}$ 
such that $F_1 (l) \leq K F_2 (Kl)$ for all $l \in \mathbb{N}$, 
and $F_1 \approx F_2$ if $F_1 \preceq F_2$ and $F_2 \preceq F_1$. 
It is clear that $\approx$ is an equivalence relation. 
\end{notn}

The following was proved as Remark 2.3 in \cite{Brad}.

\begin{propn} \label{free-gp}
Let $k,k' \in \mathbb{N}_{\geq 2}$. Then 
$\mathcal{A}_{\Gamma,F_k\setminus \lbrace 1 \rbrace} ^S 
\approx \mathcal{A}_{\Gamma,F_{k^{\prime}}\setminus \lbrace 1 \rbrace} ^S$.
\end{propn}

\begin{lem} \textnormal{\cite[Corollary 2.5]{Brad}}
Let $S_1$ and $S_2$ be finite generating sets for $\Gamma$. 
Then $\mathcal{A}_{\Gamma} ^{S_1} \approx\mathcal{A}_{\Gamma} ^{S_2}$. 
\end{lem}

For a word $w \in F_k$ we define the \emph{vanishing set} of $w$ in $\Gamma$ by
\begin{equation*}
Z(w,\Gamma) \coloneq \lbrace (g_1, ... , g_k) \in \Gamma^k : w(g_1, ... , g_k)=e \rbrace\text{.}
\end{equation*}

\begin{lem} \label{KozThoProp} \textnormal{\cite[Lemma 2.2]{KozTho}}
Let $k \geq 2$ and let $w_1 , \ldots , w_m \in F_k$ be nontrivial. 
There exists $w \in F_k$ nontrivial such that, 
for any group $\Gamma$, 
\begin{center}
$Z(w,\Gamma) \supseteq Z(w_1,\Gamma) \cup \cdots \cup Z(w_m,\Gamma)$
\end{center}
and $\lvert w \rvert \leq 16 m^2 \max_i \lvert w_i \rvert$. 
\end{lem}

\begin{coroll} \label{KozThoCoroll} \textnormal{\cite[Corollary 2.9]{Brad}}
Let $\Gamma$ be a lawless group and let $k \geq 2$. 
Then for all $l \geq 1$, there exist $g_1 , \ldots , g_k \in \Gamma$ 
such that, for all $v \in F_k$ nontrivial with $\lvert v \rvert \leq l$, 
$v (g_1 , \ldots , g_k) \neq e$.   
\end{coroll}

\begin{lem} \label{universal-unlaw}
Suppose $\Gamma$ has a shortest law in $F_k$ of length at least $l$. Let 
\begin{equation*}
W \coloneq \{ w \in F_k\backslash\{1\} : \lvert w \rvert < l \}\textnormal{.}
\end{equation*}
Then there exists $\bar{g}_1, ... , \bar{g}_k \in \Gamma^W$ where $w(\bar{g}_1, ... , \bar{g}_k) \neq e$ for all $w \in W$.
\end{lem}
\begin{proof}
We see each word $w \in F_k\backslash\{1\}$ where $\lvert w \rvert < l$ is not a law for $\Gamma$. Then for each $w \in W$ there exists $g^w_1, ... , g^w_k \in \Gamma$  where $w(g^w_1, ... , g^w_k) \neq e$. Define each $\bar{g}_i \in \Gamma^W$ by $(\bar{g}_i)_w \coloneq g^w_i$. Therefore $(w(\bar{g}_1, ... , \bar{g}_k))_w \neq e$ for all $w \in W$.
\end{proof}

Note a shortest law of $\Gamma$ is also a shortest law of $\Gamma^W$, which will be used in the following Section.

\section{Sparse wreath products}

\subsection{Construction} \label{ConstrSubsect}

Suppose $\Delta$ and $H$ are groups. Let $\Delta \textnormal{ wr } H \coloneq (\bigoplus_{H} \Delta) \rtimes H$ and $\Delta \Wr H \coloneq \Delta^{H} \rtimes H$ 
be respectively the \emph{restricted} and \emph{unrestricted} wreath product of $\Delta$ and $H$. Note that $\bigoplus_{H} \Delta  \trianglelefteq  \Delta^{H}$ and $\Delta \textnormal{ wr } H \leq \Delta \Wr H$.

\begin{lem} \label{SparseFnctLem} \textnormal{\cite[Lemma 4.1]{Brad}}
There exist increasing functions 
$p,q : \mathbb{N} \rightarrow \mathbb{N} \cup \lbrace 0 \rbrace$ 
such that $p(1)=q(1)=0$ and, for all $i,j,k,l \in \mathbb{N}$, 
\begin{itemize}
\item[(a)] For $r = p$ or $q$, 
if $r(j)-r(k)=r(l)-r(i)$ then either (i) $i=k$ and $j=l$, 
or (ii) $j=k$ and $i=l$; 

\item[(b)] If $q(j)-p(k)=q(l)-p(i)$ then $i=k$ and $j=l$. 

\end{itemize}
\end{lem}

We shall require a slight strengthening of the statement of Lemma \ref{SparseFnctLem}, 
obtained by the same proof. 

\begin{lem} \label{pq-other-def}
Let $p,q : \mathbb{N} \rightarrow \mathbb{N} \cup \lbrace 0 \rbrace$ be non-decreasing functions. Let $p(1)=q(1)=0$. Suppose for all $n \in \mathbb{N}$
\begin{itemize}
\item[(i)] $p(n+1) \geq p(n) + q(n) + 1$
\item[(ii)] $q(n+1) \geq p(n+1) + q(n) + 1$
\end{itemize}
then $p,q$ satisfy the conditions of Lemma \ref{SparseFnctLem} and $(q(n)-p(n))$ is a non-negative non-decreasing sequence.
\end{lem}
\begin{proof}
The conditions (i) and (ii) appear as $(4.1)$ and $(4.2)$ in the proof of Lemma 4.1 in \cite{Brad}. 
The subsequent proof of  Lemma 4.1 in \cite{Brad} then consists in showing that any functions $p$ and $q$ 
satisfying conditions (i) and (ii) also satisfy properties (a) and (b) above. 
We see in addition that from (i) and (ii) we have 
$q(1)-p(1)=0$ and $q(n+1)-p(n+1) \geq (p(n+1) + q(n) + 1) - p(n+1) \geq q(n) + 1$. 
\end{proof}

Henceforth let $\Delta$ be a group; let $(g_l),(h_l)$ be sequences of elements of $\Delta$; 
let $L : \mathbb{N} \rightarrow \mathbb{N}$ be a non-decreasing unbounded 
function to be determined, 
and let $p,q : \mathbb{N} \rightarrow \mathbb{N} \cup \lbrace 0 \rbrace$ be functions 
satisfying the conclusion of Lemma \ref{SparseFnctLem}. 

Let $\hat{g} , \hat{h} : \mathbb{Z} \rightarrow \Delta$ be defined by
\begin{itemize}
\item[(i)] For each $n \in \mathbb{N}$, 
$\hat{g} (p(n)) = g_{L(n)}$ and $\hat{h} (q(n)) = h_{L(n)}$; 
\item[(ii)] $\hat{g}(m)=e$ for $m \notin \im (p)$; 
\item[(iii)] $\hat{h}(m)=e$ for $m \notin \im (q)$. 
\end{itemize}

\begin{defn} \label{gamma}
Let $G = \Delta \Wr \mathbb{Z} = \Delta^{\mathbb{Z}} \rtimes \mathbb{Z}$ 
be the unrestricted wreath product of $\Delta$ and $\mathbb{Z}$, 
with the $\mathbb{Z}$-factor being generated by $t$. 
Let $S=S(L)=\lbrace \hat{g} , \hat{h} , t \rbrace$, for $\hat{g} , \hat{h} \in \Delta^{\mathbb{Z}}$ as above. 
The \emph{sparse wreath product} associated to the data $p,q,L,(g_l),(h_l)$ is 
the group $\Gamma = \Gamma(L) = \langle S(L) \rangle \leq G$.
\end{defn}

Recall that the class of \emph{elementary amenable} groups 
is the smallest class containing all finite groups and $\mathbb{Z}$, 
which is closed under subgroups; quotients; extensions and directed limits of groups. 

\begin{thm} \label{ele-am} \textnormal{\cite[Theorem 4.3]{Brad}}
Suppose $\Delta$ is elementary amenable. For any $p$, $q$, $L$, $(g_l)$, $(h_l)$ as above, 
the associated sparse wreath product $\Gamma$ is elementary amenable. 
\end{thm}

\begin{rmrk}
The relevance of Theorem \ref{ele-am} to the study of lawlessness 
is that elementary amenable groups do not contain nonabelian free subgroups. 
As such, by Theorem 1.1 of \cite{Brad}, they do not have bounded 
lawlessness growth. 
\end{rmrk}

We define the \emph{support} of $\hat{f} \in \Delta^{\mathbb{Z}}$ by $\supp (\hat{f}) \coloneq \lbrace n\in\mathbb{Z} : \hat{f}(n)\neq e \rbrace$. The following are useful properties coming from the proof of the previous Theorem.

\begin{lem} \label{fin-supp}
Let $X = \lbrace \hat{g}^{t^n} , \hat{h}^{t^n} : n\in\mathbb{Z} \rbrace$. Then
\begin{itemize}
\item[(a)] for every $\hat{f}_1 , \hat{f}_2 \in \Delta^{\mathbb{Z}}$, 
$\supp([\hat{f}_1 , \hat{f}_2]) \subseteq \supp (\hat{f}_1) \cap \supp (\hat{f}_2)$; 
\item[(b)] For all pairs of distinct elements $\hat{f}_1 , \hat{f}_2 \in X$, 
$\lvert \supp (\hat{f}_1) \cap \supp (\hat{f}_2) \rvert \leq 1$.
\end{itemize}
\end{lem}
\begin{proof}
For (a), suppose $\hat{f}_i(n)=e$ for some $i\in\{1,2\}$. Then $[\hat{f}_1 , \hat{f}_2](n)=e$.

For (b), suppose $\lvert \supp (\hat{f}_1) \cap \supp (\hat{f}_2) \rvert \geq 2$. Then there exists distinct $n,m \in \mathbb{Z}$ where $\hat{f}_i(n) \neq e$ and $\hat{f}_i(m) \neq e$ for both $i = 1$ and $i = 2$. If $\hat{f}_1 = r^{t^k}$ and $\hat{f}_2 = r^{t^l}$ for some $r \in \{\hat{g}, \hat{h} \}$, then this would contradict Lemma~\ref{SparseFnctLem} (a). Otherwise, we would be in the situation where  $\hat{f}_1 = r^{t^k}$ and $\hat{f}_2 = s^{t^l}$ for distinct $r,s \in \{\hat{g}, \hat{h} \}$, which would contradict Lemma~\ref{SparseFnctLem} (b).
\end{proof}

Henceforth suppose $\Delta$ is a lawless group. 
Using Corollary \ref{KozThoCoroll}, for each $l \geq 1$ we may take $g_l,h_l \in \Delta$ 
such that for all $v \in F_2$ nontrivial, 
if $\lvert v \rvert \leq l$ then $v (g_l , h_l) \neq e$. Under these hypotheses, 
one can prove the following.

\begin{thm} \label{slow-A} \textnormal{\cite[Theorem 4.4]{Brad}}
For every unbounded non-decreasing function $f : \mathbb{N} \rightarrow \mathbb{N}$, 
there exists $L$ such that for all $n \in \mathbb{N}$, 
$\mathcal{A}_{\Gamma(L)} ^{S(L)}(n) \preceq f(n)$. 
\end{thm}

Our main results, proved in Section \ref{MainProofsSect}, 
shall show that, under stronger assumptions on $\Delta$, $L$, $p$ and $q$ 
one can obtain a complementary lower bound on $\mathcal{A}_{\Gamma(L)} ^{S(L)}$.

\subsection{Groups not residually finite}

We recall the following well-known result of Gr\"{u}nberg \cite{Gruen}. 

\begin{thm} \label{rf-wr} 
Let $\Delta$ and $H$ be groups. Suppose $\Delta \textnormal{ wr } H$ is residually finite. Then either $\Delta$ is abelian or $H$ is finite.
\end{thm}

\begin{propn} \label{not-rf}
Suppose that $\Delta$, $g_l$, $h_l$ satisfy the assumptions preceding Theorem \ref{slow-A} above. 
Then for any unbounded nondecreasing function $L:\mathbb{N}\rightarrow\mathbb{N}$ and any $p,q$ satisfying the conclusion of Lemma \ref{SparseFnctLem}, the associated sparse wreath product $\Gamma$ is not residually finite. 
\end{propn}
\begin{proof}
Consider the non-trivial word $w \in F_2$ defined by $w(x,y) \coloneq [[x,y],[y,x^{-1}]]$. Let $j \in \mathbb{N}$ be big enough such that $L(j) \geq 16$. Then $w (g_{L(j)},h_{L(j)}) \neq e$.

Consider the elements $\bar{g}, \bar{h} \in \Gamma$ defined by 
\begin{itemize}
\item[(i)] $\bar{g} \coloneq [\hat{g}^{t^{q(j)-p(j)}},\hat{h}]$; 
\item[(ii)] $\bar{h} \coloneq [\hat{h}, (\hat{g}^{t^{q(j)-p(j)}})^{-1}]$.
\end{itemize}
Then 
$\bar{g}(q(j)) = [\hat{g}^{t^{q(j)-p(j)}},\hat{h}](q(j)) = [g_{L(j)},h_{L(j)}] \neq e$ 
(since $w(x,y)$ is of length less than $16$) 
and similarly $\bar{h}(q(j)) \neq e$. By Lemma \ref{fin-supp} we see this is the only position $\bar{g}$ and $\bar{h}$ are supported.

The group $\overline{\Delta} \coloneq \langle \bar{g}(q(j)),\bar{h}(q(j)) \rangle \leq \Delta$ is non-abelian, as $[\bar{g}(q(j)),\bar{h}(q(j))] = w (g_{L(j)},h_{L(j)}) \neq e$. By Theorem \ref{rf-wr} the group $\overline{\Delta} \textnormal{ wr } \mathbb{Z}$ is not residually finite. 
As $\lbrace \bar{g} , \bar{h} , t \rbrace \subset \Gamma$ 
is a generating set for $\overline{\Delta} \textnormal{ wr } \mathbb{Z}$
we have shown $\overline{\Delta} \textnormal{ wr } \mathbb{Z} \leq \Gamma$ so $\Gamma$ is not residually finite.
\end{proof}

\section{Controlling lawlessness growth in sparse wreath products} \label{MainProofsSect}

\subsection{Amenable groups of fast lawlessness growth}

In this Subsection we will prove Theorem \ref{intro-faster-growth}. For a group $G$ let us denote the length of a shortest law (on any number of generators) by $\alpha (G)$.

\begin{lem} \label{sl-linear} Consider the group $\PSL_2 (p)$ where $p$ is prime. Then 
\begin{equation*}
(p-1)/3 \leq \alpha (\PSL_2 (p)) \leq 8p + 6.
\end{equation*}
\end{lem}
\begin{proof}
By \cite[Theorem 2]{Had} we see that $A_1(p) = \PSL_2 (p)$ has a lower bound on a shortest law, given by $(p-1)/3$. It is a known property of $\PSL_2 (p)$ that the order of every element divides either $p-1$,$p$ or $p+1$. Therefore the word $w \in F_2 \backslash \{1\}$ defined by $w(x,y) = [[yx^{p-1}y^{-1},x^{p}],x^{p+1}]$ is a law. Note that $\lvert w \rvert = 8p + 6$ which gives us the upper bound.
\end{proof}

\begin{rmrk}
Looking at the characteristic polynomial of elements in $\SL_2 (p)$, consider the possible eigenvalues of the elements. Examining the cases of repeated eigenvalues, unique eigenvalues splitting over $\mathbb{F}_p$, or neither of the previous but eigenvalues splitting over $\mathbb{F}_{p^2}$, leads to the element having order dividing $2p$, $p-1$ and $p+1$ respectively. This proof can then be adjusted to show elements of $\PSL_2 (p)$ divide either $p$, $p-1$ or $p+1$. 
Lemmas \ref{law-with-square} and \ref{law-without-square} below could be proved replacing $\PSL_2 (p)$ with $\SL_2 (p)$ but with worse bounds.
\end{rmrk}

Let $(p_n)$ be an unbounded sequence of primes (to be determined). Consider the group
\begin{equation} \label{DeltaPSLDefn}
\Delta \coloneq \bigoplus_{n=1}^\infty \PSL_2 (p_n)^{W_n}
\end{equation}
where $W_n \coloneq \{ w \in F_2 \backslash \{ 1 \} : \lvert w \rvert < (p_n-1)/3 \}$. Note that, by the lower bound on $\alpha (\PSL_2 (p))$ in Lemma \ref{sl-linear}, the group $\Delta$ is lawless. We see $\Delta$ is also locally finite, hence elementary amenable. By Lemma \ref{sl-linear} and Lemma \ref{universal-unlaw}, there exists elements $\bar{g}_n , \bar{h}_n \in \PSL_2 (p_n)^{W_n}$ where $w(\bar{g}_n , \bar{h}_n) \neq e$ for all $w \in W_n$. 

Consider our non-decreasing function $L : \mathbb{N} \rightarrow \mathbb{N}$ still yet to be determined. We will define our sequence $(p_n)$ by each term being the smallest prime where $L(n) < (p_n-1)/3$.

Now we have elements $g_{L(n)} \coloneq \bar{g}_n$ and $h_{L(n)} \coloneq \bar{h}_n$. 
Thus given functions $p$ and $q$ satisfying the conclusion of Lemma \ref{SparseFnctLem}, we may define $\hat{g} , \hat{h} : \mathbb{Z} \rightarrow \Delta$ as in Subsection \ref{ConstrSubsect}. 
Note that for $v \in F_2 \backslash \{1\}$ where $\lvert v \rvert \leq L(n)$ we get that $v \in W_n$ hence in $\Delta \Wr \mathbb{Z}$ we have: 
\begin{equation} \label{SWPLawlessEqn}
v(\hat{g}^{t^{q(n)-p(n)}},\hat{h})(q(n)) = v(g_{L(n)}, h_{L(n)}) \neq e. 
\end{equation}
We will need the following.

\begin{thm} \label{bertrand} \textnormal{(Bertrand's postulate)} For any integer $n>1$ there exists a prime $p$ such that $n<p<2n$.
\end{thm}
Note by Bertrand's postulate we have that 
\begin{equation} \label{BertrandEqn}
3L(n)+1 < p_n < 6L(n)+2.
\end{equation}

\begin{lem} \label{law-with-square}
There is a law $\overline{w}_n \in F_2\backslash\{1\}$ for $\bigoplus_{i=1}^n \PSL_2 (p_i)^{W_i}$ where
\begin{equation*}
\lvert \overline{w}_n \rvert < n^2 (768L(n) + 352).
\end{equation*}
\end{lem}
\begin{proof}
Let $w_n \in F_2\backslash\{1\}$ be a law for $\PSL_2 (p_n)^{W_n}$ such that $\lvert w_n \rvert \leq 8p_n + 6$ as in Lemma \ref{sl-linear}. Then there exists a law $\overline{w}_n \in F_2\backslash\{1\}$ for $\bigoplus_{i=1}^n \PSL_2 (p_i)^{W_i}$ where
\begin{equation*}
\begin{aligned}
    \lvert \overline{w}_n \rvert & \leq 16 n^2 \max \{ \lvert w_i \rvert : 1 \leq i \leq n \} \text{ by Lemma \ref{KozThoProp}} \\
      & \leq 16 n^2 (8p_n + 6) \\
      & < 16 n^2 (8(6L(n)+2) + 6).
  \end{aligned}
\end{equation*}
\end{proof}

\begin{lem} \label{law-without-square}
Suppose $L:\mathbb{N} \rightarrow \mathbb{N}$ is defined such that $L(n+1) \geq 64L(n) + 29$ for all $n \in \mathbb{N}$. Then there exists a law $\overline{w}_n \in F_2\backslash\{1\}$ for $\bigoplus_{i=1}^n \PSL_2 (p_i)^{W_i}$ where
\begin{equation*}
\lvert \overline{w}_n \rvert < 3072L(n) + 1408.
\end{equation*}
\end{lem}
\begin{proof}
We construct $\overline{w}_n$ by an inductive process. We know $3L(n)+1 < p_n < 6L(n)+2$ so by Lemma \ref{sl-linear} there exists a law $w_n \in F_2\backslash\{1\}$ for $\PSL_2 (p_n)^{W_n}$ where 
\begin{equation*}
\lvert w_n \rvert \leq 8p_n + 6 < 48L(n)+22\textnormal{.}
\end{equation*}
Hence there exists $\overline{w}_{1}$  where $\lvert \overline{w}_{1} \rvert < 48L(1)+22 \leq 3072L(1) + 1408$.

Suppose we have found $\overline{w}_{n-1}$ as in the statement of the Theorem, with 
\begin{equation*}
\lvert \overline{w}_{n-1} \rvert < 3072L(n-1) + 1408 = 64 (48L(n-1)+22)\textnormal{.}
\end{equation*}
As $L(n) \geq 64L(n-1) + 29$ then $48L(n)+22 \geq 64 (48L(n-1)+22)$ hence
\begin{equation*}
\begin{aligned}
    \lvert \overline{w}_n \rvert & \leq 64 \max \{ \lvert \overline{w}_{n-1} \rvert , \lvert w_n \rvert \} \text{ by Lemma \ref{KozThoProp}} \\
      & < 64 \max \{ 64 (48L(n-1)+22), 48L(n)+22 \} \\
      & < 64 (48L(n)+22).
  \end{aligned}
\end{equation*}
\end{proof}

\begin{defn} \label{v_n}
Define the word $v_n \in F_{8}$ by
\begin{equation*}
v_n \coloneq \overline{w}_{n}([[x_1,x_2],[x_3,x_4]],[[x_5,x_6],[x_7,x_8]])
\end{equation*}
where $\overline{w}_{n}$ is a shortest law for $\bigoplus_{i=1}^n \PSL_2 (p_i)^{W_i}$. This can be seen as applying the commutator two times followed by $\overline{w}_{n}$.
\end{defn}

\begin{rmrk} \label{square-diff}
Note that using the bound on $\lvert \overline{w}_{n} \rvert$ from Lemma \ref{law-with-square} would mean there exists some $\alpha \in \mathbb{N}$ where $\lvert v_n \rvert \leq \alpha n^2 L(n)$ for all $n \in \mathbb{N}$, but using the bound from Lemma \ref{law-without-square} would mean there exists some $\alpha \in \mathbb{N}$ where $\lvert v_n \rvert \leq \alpha  L(n)$ for all $n \in \mathbb{N}$. 
\end{rmrk}

Henceforth $p,q : \mathbb{N} \rightarrow \mathbb{N} \cup \lbrace 0 \rbrace$ 
and $L : \mathbb{N} \rightarrow \mathbb{N}$ are nondecreasing functions, 
yet to be determined, with $p,q$ satisfying the conclusion of Lemma \ref{SparseFnctLem}. 
Let $\Delta$ be as in (\ref{DeltaPSLDefn}), with the sequence $(p_n)$ 
satisfying (\ref{BertrandEqn}), and let $\hat{g} , \hat{h} : \mathbb{Z} \rightarrow \Delta$ 
be as in the paragraph preceding Theorem \ref{bertrand}. 
We let $\Gamma = \Gamma(L) \leq \Delta \Wr \mathbb{Z}$ 
be the associated sparse wreath product generated by $S(L)=\lbrace \hat{g} , \hat{h} , t \rbrace$. 
Then $\Gamma$ is lawless (by (\ref{SWPLawlessEqn})) and elementary amenable (by Theorem \ref{ele-am}). 
In the remainder of this Note, we show that $\mathcal{A}_{\Gamma(L)} ^{S(L)}$ 
can be effectively controlled using our choice of $L$, $p$ and $q$.

\begin{lem} \label{cmplx-lwbnd}
Suppose $(q(n)-p(n))$ is a non-negative non-decreasing sequence and $v_n$ is defined as in Definition \ref{v_n}. Then for all $n \in \mathbb{N}$,
\begin{equation*}
\chi_{\Gamma(L)} ^{S(L)} (v_n) \geq (q(n)-p(n))/2.
\end{equation*}
\end{lem}

Note that functions $p$ and $q$ satisfying the hypothesis of Lemma \ref{cmplx-lwbnd} 
and the conclusion of Lemma \ref{SparseFnctLem} exist, by Lemma \ref{pq-other-def}. 

\begin{proof}
Let $\bar{z}=(z_1,...,z_8)$ for $z_i \in \Gamma $ be such that $\sum_{i=1} ^8 \lvert z_i \rvert_{S(L)} < (q(n)-p(n))/2$. Since $\Gamma$ is generated by $S = \lbrace \hat{g} , \hat{h} , t \rbrace$ we can express these elements as $z_i = (\theta,t^a) \in \Delta^{\mathbb{Z}} \rtimes \mathbb{Z}$ where 
\begin{equation} \label{gen-form}
    \theta = (\hat{h}^{\alpha_1})^{t^{\beta_1}}(\hat{g}^{\alpha_2})^{t^{\beta_2}}...(\hat{h}^{\alpha_{k-1}})^{t^{\beta_{k-1}}}(\hat{g}^{\alpha_k})^{t^{\beta_k}}.
\end{equation}
for some $\alpha_i , \beta_j \in \mathbb{Z}$. 
We choose an expression of the form (\ref{gen-form}) for $\theta$ minimizing the quantity 
$B_i = \max \lbrace \lvert \beta_m \rvert : 1 \leq m \leq k \rbrace$, which we will call the \textit{maximal shift} in $z_i$. Note that $B_i \leq \lvert z_i \rvert_{S(L)}$ and $\lvert a \rvert \leq \lvert z_i \rvert_{S(L)}$. Consider $z_j = (\psi,t^b)$ with its own expression as in equation (\ref{gen-form}), and its own maximal shift $B_j$. Let
\begin{equation*}
\begin{aligned}
z_{ij} \coloneq [z_i,z_j] & = (\theta,t^a)(\psi,t^b)((\theta^{-1})^{t^{-a}},t^{-a})((\psi^{-1})^{t^{-b}},t^{-b}) \\
& = (\theta\psi^{t^a},t^{a+b})((\theta^{-1})^{t^{-a}}(\psi^{-1})^{t^{-a-b}},t^{-a-b}) \\
& = (\theta \psi^{t^a}(\theta^{-1})^{t^b}(\psi^{-1}),e)\textnormal{.}
\end{aligned}
\end{equation*}
Substituting our expressions for $\theta$ and $\psi$ into $\theta \psi^{t^a}(\theta^{-1})^{t^b}(\psi^{-1})$ gives an expression for $z_{ij}$ as in equation (\ref{gen-form}). We can see that 
the maximal shift $B_{ij}$ in $z_{ij}$ satisfies
\begin{equation*}
\begin{aligned}
B_{ij} & \leq \max \{ B_i ,  B_j  + \lvert a \rvert, B_i + \lvert b \rvert, B_j \} \\
& \leq \lvert z_i \rvert_{S(L)} + \lvert z_j \rvert_{S(L)} \\
& < (q(n)-p(n))/2 \textnormal{.}
\end{aligned}
\end{equation*}
Now consider distinct $i,j,k,l \in \{1,..,8\}$ and the two elements $z_{ij},z_{kl} \in \Delta^{\mathbb{Z}} \times \{e\} = \Delta^{\mathbb{Z}}$ constructed in the same way. Let $z_{ijkl} \coloneq [z_{ij},z_{kl}] \in \Delta^{\mathbb{Z}}$ which gives an expression as in equation (\ref{gen-form}).

Remembering the way $\hat{g} , \hat{h} : \mathbb{Z} \rightarrow \Delta$ are defined, for distinct $x,y \in \mathbb{N}$ we can see $[g_{L(x)},h_{L(y)}] = [g_{L(x)},g_{L(y)}] = [h_{L(x)},h_{L(y)}] = e$ as the support of each element in the sum $\Delta = \bigoplus_{i=1}^\infty \PSL_2 (p_i)^{W_i}$ has empty intersection. It is clear $[g_{L(x)},g_{L(x)}] = [h_{L(x)},h_{L(x)}] = e$ but it is possible that $[g_{L(x)},h_{L(x)}] \neq e$.

Looking at $z_{ijkl}$ in terms of its induced expression of the form of equation (\ref{gen-form}), for $(z_{ijkl})_r \in \Delta$ to contain some $[g_{L(x)}^{\alpha_c},h_{L(x)}^{\alpha_d}]$ would require some $(\hat{g}^{\alpha_c})^{t^{\beta_c}}$ in the chosen expression (as in equation (\ref{gen-form})) of $z_{ij}$ or $z_{kl}$, and some $(\hat{h}^{\alpha_d})^{t^{\beta_d}}$ in the chosen expression of $z_{ij}$ or $z_{kl}$ where $\lvert \beta_c \rvert + \lvert \beta_d \rvert \geq q(x)-p(x)$. We know $ B_{ij} < (q(n)-p(n))/2 $ and $B_{kl} < (q(n)-p(n))/2 $ hence $\lvert \beta_c \rvert + \lvert \beta_d \rvert < q(n)-p(n)$. Therefore $(z_{ijkl})_r \in \bigoplus_{i=1}^n \PSL_2 (p_i)^{W_i}$ for all $r \in \mathbb{Z}$ and so $z_{ijkl} \in (\bigoplus_{i=1}^n \PSL_2 (p_i)^{W_i})^{\mathbb{Z}} \leq \Delta^{\mathbb{Z}}$. By the definition of $\overline{w}_{n}$ in Definition \ref{v_n}, any two elements $\mu , \phi \in \bigoplus_{i=1}^n \PSL_2 (p_i)^{W_i}$ satisfy $\overline{w}_{n}(\mu , \phi)=e$. Therefore $v_n(\bar{z})=\overline{w}_{n}(z_{1234} , z_{5678})=e$.
\end{proof}

We can now complete the proof of Theorem \ref{intro-faster-growth} from the Introduction. 

\begin{thm} \label{faster-growth}
Let $f : \mathbb{N} \rightarrow \mathbb{N}$ be an unbounded nondecreasing function. There exists an elementary amenable lawless group $\Gamma$, generated by a finite set $S$ such that for all $n\in\mathbb{N}$, 
\begin{equation*}
\mathcal{A}_{\Gamma} ^S (n) \succeq f (n).
\end{equation*}
\end{thm}
\begin{proof}
Let $L$ satisfy the conditions of Lemma \ref{law-without-square}, so there exists some $\alpha \in \mathbb{N}$ where $\lvert v_m \rvert \leq \alpha L(m)$ for all $m \in \mathbb{N}$. Let $\beta \coloneq L(2)$ and suppose $L(m) \leq \beta n \leq L(m+1)$. 
Let $p,q$ satisfy the conditions of Lemma \ref{pq-other-def} and: 
\begin{equation} \label{in-f}
(q(m)-p(m))/2 \geq f(L(m+1))
\end{equation}
for all $m \geq 2$. 
Then by Lemma \ref{cmplx-lwbnd}, $\chi_{\Gamma(L)} ^{S(L)} (v_n) \geq (q(n)-p(n))/2$.
Thus: 
\begin{equation*}
\mathcal{A}_{\Gamma(L)} ^{S(L)} (\alpha \beta n) \geq \mathcal{A}_{\Gamma(L)} ^{S(L)} (\alpha L(m)) \geq (q(m)-p(m))/2 \geq f(L(m+1)) \geq f(\beta n) \geq f(n).
\end{equation*}
As noted immediately before Lemma \ref{cmplx-lwbnd}, 
$\Gamma(L)$ is lawless and elementary amenable.
\end{proof}

\begin{rmrk} \label{square-difference}
Theorem \ref{faster-growth} could be proved without $L$ satisfying Lemma \ref{law-without-square}, by using Lemma \ref{law-with-square} instead. However, this difference in method would require $(q(n)-p(n))$ to grow too fast for the equivalence later proved in Lemma \ref{equiv+constants}.
\end{rmrk}

\subsection{Amenable groups of prescribed lawlessness growth}

In this Subsection we will prove Theorem \ref{intro-fast-f} and Theorem \ref{intro-f-slow}. 
We retain the notation and assumptions of the paragraph preceding Lemma \ref{cmplx-lwbnd}.

\begin{lem} \label{lg-1-split}
Assume $L$ satisfies the conditions of Lemma \ref{law-without-square} and $p,q$ satisfy the conditions of Lemma \ref{cmplx-lwbnd}. Then,
\begin{equation*}
\mathcal{A}_{\Gamma(L)}^{S(L)}(L(m)) \succeq q(m)-p(m).  
\end{equation*}
\end{lem}
\begin{proof}
By the bound on the length of $\overline{w}_m$ in Lemma \ref{law-without-square} there exists some constant $\alpha \in \mathbb{N}$ where $\lvert v_m \rvert \leq \alpha L(m) \leq L( \alpha m)$ for all $m \in \mathbb{N}$. By Lemma \ref{cmplx-lwbnd} we get that $\chi_{\Gamma(L)} ^{S(L)} (v_m) \geq (q(m)-p(m))/2$. Putting these results together shows $\mathcal{A}_{\Gamma(L)}^{S(L)}(L(m)) \succeq q(m)-p(m)$.
\end{proof}

\begin{lem} \label{lg-1}
Assume $L$ satisfies the conditions of Lemma \ref{law-without-square} and $p,q$ satisfy the conditions of Lemma \ref{cmplx-lwbnd}. Then,
\begin{equation*}
\mathcal{A}_{\Gamma(L)}^{S(L)}(L(m+1)) \approx q(m+1)-p(m+1).  
\end{equation*}
\end{lem}
\begin{proof}
 Owing to Lemma \ref{lg-1-split} we only need to show $\mathcal{A}_{\Gamma(L)}^{S(L)}(L(m+1)) \preceq q(m+1)-p(m+1)$. Let $w \in F_2$ be nontrivial with $\lvert w \rvert \leq L(m)$. Then 
\begin{equation*}
 w( \hat{g}^{t^{q(m)-p(m)}},\hat{h} )( q(m) ) 
 = w( g_{L(m)},h_{L(m)} ) \neq e
\end{equation*} 
hence 
\begin{equation*}
\chi_{\Gamma(L)} ^{S(L)} (w) \leq \lvert \hat{g}^{t^{q(m)-p(m)}} \rvert_{S(L)} + \lvert \hat{h} \rvert_{S(L)}
\leq 2 ( q(m)-p(m)  + 1) \text{.}
\end{equation*}
If $m \geq 2$ then $2 ( q(m)-p(m)  + 1 ) \leq 4(q(m)-p(m))$ so $\mathcal{A}_{\Gamma(L)}^{S(L)}(L(m+1)) \preceq q(m+1)-p(m+1)$.
\end{proof}

Note that $q(1)-p(1) = 0$ which is why we use $L(m+1)$ instead of $L(m)$.

\begin{rmrk} \label{up-lw}
Note the lower bound of $\mathcal{A}_{\Gamma(L)}^{S(L)}(n)$ uses words in $F_8$ while the upper bound uses words in $F_2$. This difference does not matter owing to Proposition \ref{free-gp}. 
\end{rmrk}

\begin{lem} \label{equiv+constants}
Let $p,q$ satisfy Lemma \ref{pq-other-def} and $L$ satisfy Lemma \ref{law-without-square}. Let $f : \mathbb{N} \rightarrow \mathbb{N}$ be an unbounded non-decreasing function. Suppose there exists constants $K,C \in \mathbb{N}$ where   
\begin{itemize}
\item[(i)] $K(q(Km)-p(Km))/2 \geq f(L(m+1))$
\item[(ii)] $Cf(CL(m)) \geq 2 ( q(m+1)-p(m+1) + 1) $
\end{itemize}
then there exists an elementary amenable lawless group $\Gamma$, generated by a finite set $S$ such that for all $n\in\mathbb{N}$, 
$\mathcal{A}_{\Gamma} ^S (n) \approx f (n)$.
\end{lem}
\begin{proof}
Repeating the proof of Theorem \ref{faster-growth} but instead defining $p,q$ to satisfy (i) rather than (\ref{in-f}) gives us 
\begin{equation*}
K\mathcal{A}_{\Gamma(L)} ^{S(L)} (\alpha \beta Kn) \geq K\mathcal{A}_{\Gamma(L)} ^{S(L)} (\alpha L(Km)) \geq K(q(Km)-p(Km))/2 \geq f(L(m+1)) \geq f(\beta n)
\end{equation*}
so that $\mathcal{A}_{\Gamma} ^S (n) \succeq f (n)$.

To get $\mathcal{A}_{\Gamma} ^S (n) \preceq f (n)$ we will repeat the proof of \cite[Theorem 4.4]{Brad} but instead requiring $f$ to satisfy (ii). Let $w \in F_2$ be nontrivial with $\lvert w \rvert \leq L(m)$. Then 
\begin{equation*}
 w( \hat{g}^{t^{q(m)-p(m)}},\hat{h} )( q(m) ) 
 = w( g_{L(m)},h_{L(m)} ) \neq e
\end{equation*} 
hence 
\begin{equation*}
\chi_{\Gamma(L)} ^{S(L)} (w) \leq \lvert \hat{g}^{t^{q(m)-p(m)}} \rvert_{S(L)} + \lvert \hat{h} \rvert_{S(L)}
\leq 2 ( q(m)-p(m)  + 1) \text{.}
\end{equation*}
Then for $n \leq L(1)$ we have $\mathcal{A}_{\Gamma(L)} ^{S(L)} (n) \leq \mathcal{A}_{\Gamma(L)} ^{S(L)} (L(1)) \leq 2 ( q(1)-p(1) + 1) = 2$ and for $L(m)\leq n \leq L(m+1)$ we have 
\begin{equation*}
\mathcal{A}_{\Gamma(L)} ^{S(L)} (n) \leq \mathcal{A}_{\Gamma(L)} ^{S(L)} (L(m+1)) \leq 2 ( q(m+1)-p(m+1)  + 1) \leq Cf(CL(m)) \leq Cf(Cn) 
\end{equation*}
hence $\mathcal{A}_{\Gamma} ^S (n) \preceq f (n)$ and so $\mathcal{A}_{\Gamma} ^S (n) \approx f (n)$.
\end{proof}

\begin{thm} \label{fast-f}
 Let $f : \mathbb{N} \rightarrow \mathbb{N}$ be an unbounded non-decreasing function. Suppose there exists a constant $M \in \mathbb{N}$ where $f(Mn) \geq 9f(n)$ for all $n \in \mathbb{N}$. Then there exists an elementary amenable lawless group $\Gamma$, generated by a finite set $S$ such that for all $n\in\mathbb{N}$, 
$\mathcal{A}_{\Gamma} ^S (n) \approx f (n)$.
\end{thm}
\begin{proof}
Consider an integer $T \geq M$. Then $f(Tn) \geq f(Mn) \geq 9f(n)$ as $f$ is non-decreasing. Pick $T$ large enough so that Lemma \ref{law-without-square} is satisfied by $L(m) \coloneq T^m$.

Define $p(m) \coloneq f(T^{m+1})$ and $q(m) \coloneq 3f(T^{m+1})$ for $m \geq 2$ and $p(1)=q(1)=0$. We can see $q(m+1) = 3p(m+1)$ and as $f(Tn) \geq 9f(n)$ we also have $p(m+1) \geq 3q(m)$. Hence these functions satisfy Lemma \ref{pq-other-def}.

We will show these functions satisfy Lemma \ref{equiv+constants} by taking $K \coloneq 1$ and $C \coloneq 4T^2$. We see that $K(q(Km)-p(Km))/2 = f(T^{m+1}) = f(L(m+1))$ so (i) is satisfied. We also see that $Cf(CL(m)) = 4T^2f(4T^{m+2}) \geq 4f(T^{m+2})+2 = 2 ( q(m+1)-p(m+1) + 1 )$ so (ii) is satisfied. Therefore $\mathcal{A}_{\Gamma} ^S (n) \approx f (n)$.
\end{proof}

We can now complete the proof of Theorem \ref{intro-fast-f} from the Introduction.

\begin{thm} \label{fast-cor}
Let $f : \mathbb{N} \rightarrow \mathbb{N}$ be an unbounded non-decreasing function. Suppose there exists some integer $M \geq 2$ where $f(Mn) \geq Mf(n)$ for all $n \in \mathbb{N}$. Then there exists an elementary amenable lawless group $\Gamma$, generated by a finite set $S$ such that for all $n\in\mathbb{N}$, $\mathcal{A}_{\Gamma} ^S (n) \approx f (n)$.
\end{thm}
\begin{proof}
There exists some $k \in \mathbb{N}$ where $M^k \geq 9$. Then 
\begin{equation*}
9f(n) \leq M^kf(n) \leq M^{k-1}f(Mn)\leq M^{k-2}f(M^2n) \leq ... \leq f(M^kn)
\end{equation*}
which satisfies the conditions of Theorem \ref{fast-f}.
\end{proof}

\begin{ex} \label{fast-ex}
For every $\lambda \in [1,\infty)$, the function $f(n) = \lceil n^{\lambda} \rceil$ 
satisfies the conditions of Theorem \ref{fast-cor} using any integer $M \geq 2$. 
\end{ex}

\begin{thm} \label{f-slow}
Let $f : \mathbb{N} \rightarrow \mathbb{N}$ be an unbounded non-decreasing function. Suppose there exists a function $g : \mathbb{N} \rightarrow \mathbb{N}$ where $fg(n)=n$ for all $n \in \mathbb{N}$, and a constant $M \in \mathbb{N}$ where $g(Mn) \geq 93g(n)$ for all $n \in \mathbb{N}$. Then there exists an elementary amenable lawless group $\Gamma$, generated by a finite set $S$ such that for all $n\in\mathbb{N}$, 
$\mathcal{A}_{\Gamma} ^S (n) \approx f (n)$. 
\end{thm} 
\begin{proof}
As $fg(n)=n$ for all $n \in \mathbb{N}$ we can see that $g$ is injective, hence unbounded. If $n_1 < n_2$ and $g(n_2) < g(n_1)$ then $n_2 = fg(n_2) < fg(n_1) = n_1$ which is a contradiction, hence $g$ is also non-decreasing.

Let $T \in \mathbb{N}$ be a constant such that $T \geq M$ and $T \geq 9$. Define $L : \mathbb{N} \rightarrow \mathbb{N}$ as $L(m) \coloneq g(T^m)$. As $g$ is non-decreasing, we see 
\begin{equation*}
L(m+1) = g(T^{m+1}) \geq g(MT^m) \geq 93g(T^m) = 93L(m) \geq 64L(m) + 29    
\end{equation*}
so this $L$ satisfies Lemma \ref{law-without-square}.

Define $p(m) \coloneq T^m$ and $q(m) \coloneq 3T^m$ for $m \geq 2$ and $p(1)=q(1)=0$. We can see $q(m+1) = 3p(m+1)$ and as $T \geq 9$ we have $p(m+1) \geq 3q(m)$ so these functions satisfy the conditions of Lemma \ref{pq-other-def}.

We will show these functions satisfy Lemma \ref{equiv+constants} by taking $K \coloneq 2$ and $C \coloneq 5T$. We see that $K(q(Km)-p(Km))/2 = 2T^{2m}\geq T^{m+1} = f(L(m+1))$ so (i) is satisfied. We also see that 
\begin{equation*}
Cf(CL(m)) \geq Cf(L(m)) = 5T^{m+1} \geq 4T^{m+1} + 2 = 2 ( q(m+1)-p(m+1) + 1 )
\end{equation*}
so (ii) is satisfied. Therefore $\mathcal{A}_{\Gamma} ^S (n) \approx f (n)$.
\end{proof}

We can now complete the proof of Theorem \ref{intro-f-slow} from the Introduction.

\begin{thm} \label{slow-ex}
Let $f : \mathbb{N} \rightarrow \mathbb{N}$ be a surjective unbounded non-decreasing function. Suppose there exists some integer $M \geq 2$ where $f(Mn) \leq Mf(n)$ for all $n \in \mathbb{N}$. Then there exists an elementary amenable lawless group $\Gamma$, generated by a finite set $S$ such that for all $n\in\mathbb{N}$, $\mathcal{A}_{\Gamma} ^S (n) \approx f (n)$.
\end{thm}
\begin{proof}
Define $f_* (n) \coloneq \min ( f^{-1}(n) )$ and $f^* (n) \coloneq \max ( f^{-1}(n) )$. Suppose $f(x)=y$. Then
\begin{equation*}
\begin{aligned}
Mf_*(y) = Mf_*f(x)  
& \leq Mx \\
&  \leq  f^*f(Mx) \\
& \leq f^*(Mf(x)) = f^*(My) \leq f_*(My + 1) \leq f_*((M+1)y)\textnormal{.}
\end{aligned}
\end{equation*}
There exists some $k \in \mathbb{N}$ where $M^k \geq 93$. Then
\begin{equation*}
93f_*(n) \leq M^kf_*(n) \leq f_*((M+1)^kn)
\end{equation*}
hence $f_*$ satisfies the conditions of $g$ in Theorem \ref{f-slow}.
\end{proof}

\begin{ex} \label{slow-ex-exp}
For every $\mu \in (0,1]$ the function $f(n) = \lceil n^{\mu} \rceil$ satisfies: 
\begin{equation*}
f(Mn) = \lceil (Mn)^{\mu} \rceil \leq \lceil n^{\mu} M \rceil \leq \lceil M \rceil \lceil n^{\mu} \rceil  = Mf(n)
\end{equation*}
for any integer $M \geq 2$. 
Hence this function satisfies the conditions of Theorem \ref{slow-ex}.
\end{ex}

\end{document}